\documentclass[12pt]{amsart}
\usepackage{amsfonts}
\usepackage{amssymb}
\usepackage{bm}

 \newtheorem{theorem}{Theorem}[section]
 
 \newtheorem{conjecture}[theorem]{Conjecture}
 \newtheorem{lemma}[theorem]{Lemma}
 
 \theoremstyle{definition}

 \newtheorem{remark}[theorem]{Remark}

\numberwithin{equation}{section}

\renewcommand{\b}{\hm\cdot}

\begin{document}
\title[The Erd\H{o}s-Ginzburg-Ziv theorem constant of finite groups]
{The Erd\H{o}s-Ginzburg-Ziv theorem constant of finite groups}

\author{Yang Zhao}
\address{School of Mathematical Sciences\\ Tiangong University\\
Tianjin 300387, P.R.
China}
\email{Zhyang202412@163.com}

\author[G. Wang]{Guoqing Wang*}
\address{School of Mathematical Sciences\\ Tiangong University\\
Tianjin 300387, P.R.
China}
\email{gqwang1979@aliyun.com}\thanks{*Corresponding author}

\begin{abstract} Let $G$ be a multiplicatively written finite group of order $n$. The Erd\H{o}s-Ginzburg-Ziv Theorem constant of the group $G$, denoted $\mathsf E(G)$, is defined as the smallest positive integer $\ell$ with the following property: for any given sequence $(g_1,\ldots,g_{\ell})$ over $G$, there exist $n$ distinct integers  $i_1,\ldots,i_n\in \{1,\ldots,\ell\}$ such that the product of
$g_{i_1},\ldots,g_{i_n}$, in some order, is the identity element of $G$. The Erd\H{o}s-Ginzburg-Ziv Theorem constant originates from the celebrated additive theorem proved by Erd\H{o}s, Ginzburg and Ziv in 1961, which amounts to proving $\mathsf E(G)\leq 2|G|-1$ holds in case that $G$ is abelian. It is also well-known that $\mathsf E(G)=2|G|-1$ holds for all finite cyclic groups.
In 2010, Gao and Li [J. Pure Appl. Algebra] conjectured that $\mathsf E(G)\leq \frac{3|G|}{2}$ for every finite non-cyclic group $G$. In this paper, we confirm the conjecture for all non-cyclic groups $G$ whose order is not divisible by four, and characterize the groups achieving the equality $\mathsf E(G)=\frac{3|G|}{2}$ as those with a cyclic subgroup of index two.
\end{abstract}

\date{}

\maketitle

\noindent {\footnotesize {\it Keywords}: Erd\H{o}s-Ginzburg-Ziv theorem; Davenport constant; Product-one sequence; Metacyclic group} \\
\noindent {\footnotesize {\it 2020 Mathematics Subject Classifications}: Primary 11B75; Secondary 11P70}

\section{Introduction and main results}

Let $G$ be a finite group written multiplicatively with the operation $*$. The identity element $1_G$ of $G$ is denoted by $1$ for brevity, and this usage will not be ambiguous in context.
Let $S=g_1\bm\cdot \ldots\bm\cdot g_{\ell}$ be a sequence over $G$ with length $\ell$. We say $S$ is a {\sl product-one sequence} provided that the product of all its elements, in some order, equals $1$; that is,
there exists a permutation $\tau$ of $\{1,2,\ldots,\ell\}$ such that $g_{\tau (1)}*\cdots *g_{\tau (\ell)}=1$.
Recall that the Erd\H{o}s-Ginzburg-Ziv Theorem constant $\mathsf E(G)$ of $G$ is defined as the smallest positive integer $\ell$ such that every sequence of length $\ell$ over $G$ contains a product-one subsequence of length $|G|$. This invariant $\mathsf E(\cdot)$ originates from the celebrated additive theorem proved by Erd\H{o}s, Ginzburg and Ziv \cite{EGZ1961} in 1961, known as {\sl Erd\H{o}s-Ginzburg-Ziv Theorem}, which asserts that in case that $G$ is a finite {\sl abelian} group, every sequence over $G$ of length $2|G|-1$ must have a {\sl product-one sequence} of length $|G|$. In other words, the Erd\H{o}s-Ginzburg-Ziv amounts to proving that $\mathsf E(G)\leq 2|G|-1$ when $G$ is a finite abelian group.
One thing worth noting is that the Erd\H{o}s-Ginzburg-Ziv Theorem, together with the Davenport constant, forms a cornerstone of Zero-Sum Theory (see \cite{GG2006} for a survey), a subfield of Combinatorial Number Theory. The Davenport constant focuses on the minimal length of a sequence guaranteeing a nonempty product-one subsequence. In contrast,  the Erd\H{o}s-Ginzburg-Ziv theorem  constant specifically targets product-one subsequences of length exactly equal to the order of the group, which reflects a striking combinatorial rigidity of finite abelian groups.
Since then, numerous results in this field have been obtained, and a wealth of applications and connections to other fields--especially to Factorization Theory in Algebra \cite{GeK,GeR}, have been established.

Although the primary focus of zero-sum theory was initially on finite abelian groups, research in this area
has never been restricted to the abelian setting alone. We mention
here some results concerning bounds on $\mathsf E(G)$, and we refer the readers to \cite{ABR,Bass2007,BLMR,  GL2008,GJM,HZ2019,ORZZ,QL,ZG2005} for further results on $\mathsf E(G)$ for certain specific classes of groups. In 1976, Olson \cite{O} proved that $\mathsf E(G)\leq 2|G|-1$ for all finite groups. In 1984, Yuster and Peterson \cite{YP1984} showed that $\mathsf E(G)\leq 2|G|-2$ when $G$ is a non-cyclic finite solvable group. Later, Yuster \cite{YP1988} improved this bound to $\mathsf E(G)\leq 2|G|-r$ provided that $|G|\geq 600((r-1)!)^2$.
Building on these and other results, Gao and Li \cite{GL2010}   proposed the following conjecture in 2010 concerning the upper bound of the Erd\H{o}s-Ginzburg-Ziv Theorem constant.

\begin{conjecture}\label{EGZconjecture} \cite[Conjecture 3]{GL2010}
Let $G$ be a finite non-cyclic group. Then $\mathsf E(G)\leq \frac{3|G|}{2}$.
\end{conjecture}

Gao, Li and Qu \cite{GLQ} in 2023 proved Conjecture \ref{EGZconjecture} holds for all groups of odd order. Godara, Joshi and Mazumdar \cite{GJM} established that Conjecture \ref{EGZconjecture} holds for group of order $2p^{t}$, where $p$ is a prime and the subgroup of $G$ of order $p^{t}$ is abelian.
In this paper, we extend their results on this upper bound conjecture and resolve the corrsponding inverse problem. The main result of this paper is as follows.

\begin{theorem}\label{maintheorem} Let $G$ be a finite non-cyclic group such that $4\nmid |G|$.
Then $\mathsf E(G)\leq \frac{3|G|}{2}$, and moreover, the equality $\mathsf E(G)=\frac{3|G|}{2}$ holds if and only if $G$ has a cyclic subgroup of index two.
\end{theorem}

\section{Notation and Preliminaries}

For real numbers $a,b\in \mathbb{R}$, we set $[a,b ]=\{x\in \mathbb{Z}: a\leq x\leq b\}$.
Let $G$ be a finite group written multiplicatively, with operation $*$.
For any subsets $A$ of $G$, we denote $\langle A \rangle $ to be the subgroup  of $G$ generated by $A$.
For any integer $n\geq 2$, we denote by $S_n$ the symmetric group on $n$
letters
 (i.e., the group of permutations of $\{1,\ldots,n\}$) and by $C_n$ the cyclic group of order $n$.

The notation and terminology for sequences are consistent with \cite{GG2013}, \cite[Chapter 5]{GeK} and \cite{Gr}.  By a {\sl sequence} over a group $G$, we mean a finite multiset (i.e., a finite, unordered collection allowing repeated elements) over $G$.  We view such sequences as elements of the free abelian monoid $\mathcal{F}(G)$. We denote multiplication in $\mathcal{F}(G)$ by the bold symbol $\bm\cdot$ (distinct from the group multiplication $*$ of $G$), and we use brackets for all exponentiation in $\mathcal{F}(G)$.
A sequence $S \in \mathcal F(G)$ can be written as $S= g_1  \bm \cdot g_2 \bm \cdot \ldots \bm\cdot g_{\ell},$ where $|S|= \ell$ is the {\it length} of $S$. For any element $g \in G$, let $\mathsf v_g(S) = |\{ i\in [1, \ell] : g_i =g \}|\,$ denote the {\it multiplicity} of $g$ in $S$. Let $\mathsf h(S)=\max\limits_{g\in G} \mathsf v_g(S)$ be the {\it height} (maximum multiplicity) of $S$. A sequence $T \in \mathcal F(G)$ is called a {\it subsequence} of $S$, denoted by $T \mid S$, if  $\mathsf v_g(T) \le \mathsf v_g(S)$ for all $g\in G$.
Denote by $S \bm\cdot T^ {[-1]}$  the subsequence of $S$ obtained by removing the terms of $T$ from $S$.  If $S_1, S_2 \in \mathcal F(G)$, then $S_1 \bm\cdot S_2 \in \mathcal F(G)$ denotes the sequence satisfying that $\mathsf v_g(S_1 \bm\cdot S_2) = \mathsf v_g(S_1 ) + \mathsf v_g( S_2)$ for all $g \in G$. For convenience, we  write
\begin{center}
 $g^{[k]} = \underbrace{g \bm\cdot \ldots \bm\cdot g}_{k} \in \mathcal F(G)\quad$
\end{center}
for any $g \in G$ and any integer $k\geq  0$.
By $$g*S=\mathop{\bullet}\limits_{a\mid S} (g*a)$$ we mean the sequence obtained from $S$ by left-translating all terms of $S$ by $g$. For any subset $X$ of $G$, we use $\mathcal{F}(X)$ to denote the submonoid of $\mathcal{F}(G)$ generated by $X$. We write $S\in \mathcal{F}(X)$ to mean that all terms of the sequence $S$ belong to the set $X$.

Let $\varphi:G\rightarrow G'$ be a group homomorphism. Then $\varphi$ can be extended to a monoid homomorphism $\Phi:\mathcal{F}(G)\rightarrow \mathcal{F}(G')$, given by $\Phi:S\mapsto \mathop{\bullet}\limits_{g\mid S} \varphi(g)$. For convenience, we shall still use $\varphi$ to denote the extended monoid homomorphism $\Phi$. Hence,  $\varphi(S)$ is a sequence over $G'$ with length $|\varphi(S)|=|S|$.
For the sequence $S= g_1 \b g_2 \bm \cdot \ldots \bm\cdot g_{\ell} \in \mathcal F(G)$, we denote $$\pi (S) = \{g_{\tau(1)}*\cdots *g_{\tau(\ell)}: \tau \mbox{ is a permutation of } [1, \ell] \}$$ to be the {\it set of products} of $S$,  and let $\Pi_n(S) = \bigcup _{T\mid S,\ |T| = n}\pi(T)$ be the {\it set of all $n$-products} of $S$, and let $\Pi(S) = \bigcup_{1 \le n \le \ell}\Pi_n(S)$ be the {\it set of all subsequence products} of $S$.
A sequence $S\in \mathcal{F}(G)$ is called
\begin{itemize}
\item[$\bullet$]  {\it product-one} if $1 \in \pi(S)$;
\item[$\bullet$] {\it product-one free} if $1\not\in \Pi(S)$;
\item[$\bullet$] {\it $|G|$-product-one} if $S$ is a product-one sequence of length $|G|$.
\end{itemize}
Note that by convention, the empty sequence is regarded as a product-one sequence.
The {\sl small Davenport constant} of $G$, denoted $\mathsf d(G)$, is the maximal length of a product-one free sequence over $G$, i.e., the largest integer $\ell$ such that there exists a sequence $S$ over $G$ with $1\notin \Pi(S)$.

\section{Proof of Theorem \ref{maintheorem}}

To prove Theorem\ref{maintheorem}, we shall need the following lemmas.

\begin{lemma}\cite[Corollary 1.3]{QLT}\label{QLT} Let $G$ be a finite non-cyclic group and $p$ the smallest prime divisor of $|G|$. Then $\mathsf d(G)\leq \frac{|G|}{p}+p-2$.
\end{lemma}

\begin{lemma}\cite[Lemma 4]{ZG2005}\label{lowerbound} Let $G$ be a finite group. Then $\mathsf E(G)\geq \mathsf d(G)+|G|$.
\end{lemma}

\begin{lemma}\cite[Theorem 1]{G1996} \label{Gao}
Let $G$ be a finite abelian group. Then $\mathsf E(G)=\mathsf d(G)+|G|$.
\end{lemma}

\begin{remark}\label{Remark} By \cite[Theorem 1]{Olson2}, we know the values of $\mathsf d(G)$ for the following two types of group $G$. Combined with Lemma \ref{Gao}, we know if $G\cong C_3\times C_3$ then $\mathsf d(G)=4$ and
$\mathsf E(G)=13=\frac{3|G|-1}{2}$;
and if $G\cong C_3\times C_6$ then $\mathsf d(G)=7$ and $\mathsf E(G)=25=\frac{3|G|}{2}-2$.
\end{remark}

\begin{lemma}\cite[Lemma 5]{GL2010}\label{induction} Let $H$ be a normal subgroup of a finite group $G$. Suppose that $\mathsf E(G)\leq c|H|-1$ for some given constant $c$ with $1<c\leq 2$. Then $\mathsf E(G)\leq c|G|-1$.
\end{lemma}

\begin{lemma} \cite[Theorem 3.9 in Chapter 2]{Xu} \label{knowngroup}
Let $G$ be a group of order $|G|=2m$ with $m$ odd. Then $G$ has a subgroup of index $2$.
\end{lemma}

\begin{lemma}\cite[Theorem 1.2]{GLQ}\label{GLQ} Let $G$ be a non-cyclic group of odd order. If $G\not\cong C_3\times C_3$, then $\mathsf E(G)\leq \frac{3|G|-3}{2}$.
\end{lemma}

\begin{lemma}\cite[Theorem 1.2]{ORZZ}\label{ZQZk} Let $G=C_m \rtimes C_2$ be a metacyclic group of order $2m$ with $m\geq 3$. Then $\mathsf E(G)=\frac{3|G|}{2}$.
\end{lemma}

\begin{lemma}\cite[Theorem 1.9]{GZ2006}\label{structure} Let $G\cong C_n\times C_n$ with $n\geq 2$. Let $S$ be a sequence over $G$ of length $|S|=n^2+2n-3$, and let $g\in G$ with $\mathsf v_g(S)=\mathsf h(S)$. If $S$ has no product-one subsequence of length $n^2$, then $S$ has one of the following two forms (by rearranging the subscripts if necessary):
\begin{itemize}
\item[(i)] $g^{[\mathsf h(S)]}\bm\cdot a^{[mn-1]}\bm\cdot T$ with $|T|=n-1$, where $(g^{-1}*a)^{[n-1]}\bm\cdot (g^{-1}*T)$ is a product-one free sequence and $\mathsf h(S)=n^2+n-mn-1$.
\item[(ii)] $g^{[\mathsf h(S)]}\bm\cdot a^{[mn-1]}\bm\cdot b^{[tn-1]}$, where $(g^{-1}*a)^{[n-1]}\bm\cdot (g^{-1}*b)^{[n-1]}$ is a product-one free sequence and $\mathsf h(S)+mn+tn=n^2+2n-1$.
\end{itemize}
\end{lemma}

\begin{remark}\label{remark conclusion} In Lemma \ref{structure}, we can observe that $g\nmid T$, $a\nmid T$ when $S$ has Form (i),
and that $g,a,b$ are three distinct elements when $S$ has form (ii).
Therefore, we have $\mathsf h(S)\equiv -1\pmod n$ and $\mathsf h(S)\geq \frac{n}{3}$ for both forms.
\end{remark}

\begin{lemma}\label{S3C3} Let $G$ be a finite group isomorphic to $S_3\times C_3$ or $(C_3\times C_3)\rtimes_{-1} C_2$. Then $\mathsf E(G)\leq \frac{3|G|}{2}-1=26$.
\end{lemma}

\begin{proof}
Note that
\begin{equation}\label{equation G=structure}
G=\langle x,y,z | x^2=y^3=z^3=1, x^{-1}yx=y^{-1}, x^{-1}zx=z, y^{-1}zy=z\rangle
\end{equation}
or
\begin{equation}\label{equation G2=structure}
G=\langle x,y,z | x^2=y^3=z^3=1, x^{-1}yx=y^{-1}, x^{-1}zx=z^{-1}, y^{-1}zy=z\rangle,
\end{equation}
according to $G\cong S_3\times C_3$ or $G\cong(C_3\times C_3)\rtimes_{-1} C_2$, respectively.
  Let $H=\langle y,z\rangle$. We see \begin{equation}\label{equation H iso C3timesC3}
H\cong C_3\times C_3.
\end{equation} Denote $\varphi:G\longrightarrow G\diagup H$ to be the canonical epimorphism of $G$ onto the quotient group
\begin{equation}\label{equation G/HcongC2}
G\diagup H\cong C_2.
\end{equation}
 Let $S=g_1\bm\cdot g_2\bm\cdot\ldots\bm\cdot g_s$ be a sequence over $G$ of length $s=26$. It suffices to show that $S$ contains a product-one subsequence of length $|G|=18$.
We assume to the contrary that $S$ has no product-one subsequences of length $18$.

Let
\begin{equation}\label{equation factorization of Sinto1tot}
S=T_1\bm\cdot T_2\bm\cdot\ldots\bm\cdot T_t\bm\cdot W
\end{equation}
 be an arbitrary factorization of $S$ into subsequences
with $t$ in $[0,13]$ being {\bf maximal} such that \begin{equation}\label{equation length Ti=2}
|T_i|=|G\diagup H|=2,
\end{equation}
and
\begin{equation}\label{equation pi(Ti)subsetker=H}
\pi(T_i)\subset \ker \varphi=H,
\end{equation}
 where $i\in [1,t]$.
Since $\varphi(S)$ is a sequence over the quotient group $G\diagup H$ of length $|\varphi(S)|=|S|=26=12\times |G\diagup H|+\mathsf E(G\diagup H)-1$, it follows from the maximality of $t$ and the very definition of $\mathsf E(G\diagup H)$ that $t=12$ or $t=13$,
and $|W|=2$ or $|W|=0$, respectively.
We take an arbitrary element
\begin{equation}\label{equation betai in pi(Ti)}
\gamma_i\in \pi(T_i) \mbox{ for each } i\in [1,t].
\end{equation}
Then
\begin{equation}\label{equation tilde(S)}
\tilde{S}=\gamma_1\bm\cdot\gamma_2\bm\cdot\ldots\bm\cdot\gamma_{t}\in \mathcal{F}(H).
\end{equation}
By \eqref{equation length Ti=2} and \eqref{equation betai in pi(Ti)}, we derive that $\tilde{S}$ has no product-one subsequences of length $9=|H|$, since otherwise, we shall obtain a product-one subsequence of $S$ with length 18 which is absurd. By \eqref{equation H iso C3timesC3} and Remark \ref{Remark}, we see $\mathsf E(H)=13$. Combined with \eqref{equation tilde(S)}, we have that in \eqref{equation factorization of Sinto1tot}, $$t=12$$ and $|W|=2$.
We may assume without loss of generality that
\begin{equation}\label{equation Ti=g2i-1g2i}
T_i=g_{2i-1}\cdot g_{2i} \mbox{ and }\gamma_i=g_{2i-1}*g_{2i} \mbox{ for each } i\in [1,12], \mbox{ and } W=g_{25}\cdot g_{26}.
\end{equation}

Since $G\diagup H\cong C_2$ and $\pi(W)\not\subset H$, we may suppose that
\begin{equation}\label{equation g25notinH}
g_{25}\notin H \mbox{ and } g_{26}\in H.
 \end{equation}
Since $|\tilde{S}|=t=12=3^2+2\times 3-3$, by Lemma \ref{structure} and Remark \ref{remark conclusion}, we conclude that
\begin{equation}\label{equation h(S)=8or5}
\mathsf h(\tilde{S})=8 \mbox{ or } \mathsf h(\tilde{S})=5,
\end{equation}
 and furthermore,
  \begin{equation}\label{equation either or}
 \mbox{ either } (i) \ \tilde{S}=\alpha^{[8]}\bm\cdot \beta^{[2]}\bm\cdot \mathcal A, \mbox{ or } (ii) \ \tilde{S}=\alpha^{[5]}\bm\cdot \beta^{[5]}\bm\cdot  \mathcal A,
  \end{equation} respectively,
where
  \begin{equation}\label{equation alpha beta in H}
\alpha, \beta\in H \mbox{ and }  \mathcal A \in \mathcal{F}(H)
   \end{equation}
 with $|\mathcal A|=2$, and $(\alpha^{-1}*\beta)^{[2]}\bm\cdot (\alpha^{-1}*\mathcal A)$ is a product-one free sequence, i.e.,
 \begin{equation}\label{equation 1notin prod a longsequence}
 1\notin \prod\left((\alpha^{-1}*\beta)^{[2]}\bm\cdot (\alpha^{-1}*\mathcal A)\right).
 \end{equation}
Combining \eqref{equation tilde(S)}, \eqref{equation h(S)=8or5} and \eqref{equation either or}, we may assume that (rearrange subscripts if necessary)
\begin{equation}\label{equation gammai=alpha1toh}
\gamma_i=\alpha \mbox{ for } i\in[1,\mathsf h(\tilde{S})],
\end{equation}
 \begin{equation}\label{equation gammaj=beta}
 \gamma_j=\beta \mbox{ for } j\in[\mathsf h(\tilde{S})+1, 10],
 \end{equation} and $\gamma_{11}\cdot \gamma_{12}=\mathcal A$.
Combined with \eqref{equation H iso C3timesC3}, \eqref{equation alpha beta in H} and \eqref{equation 1notin prod a longsequence},  we have
 \begin{equation}\label{equation gamma1112notin}
\gamma_{11}, \gamma_{12}\notin \{\alpha,\beta\}.
\end{equation}

Now we claim that we can make the following assumption.

\noindent {\bf Assumption A.} $\alpha=y^{c} \mbox{ for some } c\in [0,2]$ in case that the group $G$ is given as \eqref{equation G=structure}.

{\sl Argument of Assumption A}.  Since $\alpha\in H$, we have $\alpha=y^{c}z^{d}$ with $c,d\in [0,2]$. Clearly, there exists an integer $e\in [0,2]$ such that $3\mid (d+2e)$. Let $S'=g_1'\bm\cdot g_2'\bm\cdot \ldots\bm\cdot g_{26}'$, where $g_i'=g_i*z^{e}$ for $i\in [1,26]$. Since $z^{e}$ belongs to the center  $Z(G)$ of the group $G$, we have that $S$ has a $|G|$-product-one subsequence if and only if $S'$ has a $|G|$-product-one subsequence. Combined with \eqref{equation Ti=g2i-1g2i} and \eqref{equation gammai=alpha1toh},  we
 may assume without loss of generality that $d=0$ and thus, the assumption holds. \qed

By  \eqref{equation G/HcongC2}, \eqref{equation pi(Ti)subsetker=H} and \eqref{equation Ti=g2i-1g2i}, we derive that $g_{2i-1}\in H$ if and only if $g_{2i}\in H$ for each $i\in [1,12]$. That is,
\begin{equation}\label{equation Ti out or in}
T_i\in \mathcal{F}(G\setminus H) \mbox{ or } T_i\in \mathcal{F}(H), \mbox{ for each } i\in [1,12].
\end{equation}
To proceed, we need to group all these subsequences $T_1,\ldots,T_{\mathsf h(\tilde{S})}$ in \eqref{equation factorization of Sinto1tot} as follows. Without loss of generality, we may suppose that
\begin{equation}\label{equation T1to Tr1out G}
T_1\bm\cdot\ldots\bm\cdot T_{r}\in \mathcal{F}(G\setminus H),
\end{equation}
and
$T_{r+1}\bm\cdot\ldots\bm\cdot T_{\mathsf h(\tilde{S})}\in \mathcal{F}(H)$,
where
\begin{equation}\label{equation r1in}
0\leq r\leq h(\tilde{S}).
\end{equation}

Now we show that
\begin{equation}\label{equation commutates if iin1toh(S)}
g_{2i}*g_{2i-1}=g_{2i-1}*g_{2i}\mbox{ for each } i\in [1, \mathsf h(\tilde{S})].
\end{equation}
To prove \eqref{equation commutates if iin1toh(S)}, we suppose to the contrary that there exists some $k\in [1,\mathsf h(\tilde{S})]$ such that $g_{2k}*g_{2k-1}\neq g_{2k-1}*g_{2k}$. Combined with \eqref{equation Ti=g2i-1g2i} and \eqref{equation gammai=alpha1toh}, we have $g_{2k-1}*g_{2k}=\gamma_k=\alpha$ and $g_{2k}*g_{2k-1}\neq \alpha$.  Combined with \eqref{equation either or}, we see that $\mathsf v_{\alpha}\left(\tilde{S}\cdot \gamma_k^{[-1]} \cdot (g_{2k}*g_{2k-1})\right)\in\{7,4\}$. Hence, the sequence $\tilde{S}\cdot \gamma_k^{[-1]} \cdot (g_{2k}*g_{2k-1})$ does not fall into any one of the two forms given in \eqref{equation either or}.
By the arbitrariness of choosing $\gamma_i$ in \eqref{equation betai in pi(Ti)}, we derive a contradiction. This proves \eqref{equation commutates if iin1toh(S)}.

\medskip

Then we shall divide the proof into two cases according to $r\geq 1$ or $r=0$, where $r$ is given as \eqref{equation r1in}.  \\

\noindent {\bf Case 1.} $r\geq 1$. \\

We first show that \begin{equation}\label{equation g1====g2r-1}
g_i=g_{25} \mbox{ for each } i\in [1,2r].
\end{equation}
Suppose to the contrary that there exists some $k\in [1,r]$ such that $g_{2k-1}\neq g_{25}$ or $g_{2k}\neq g_{25}$.
Say $g_{2k-1}\neq g_{25}$.  Combined with \eqref{equation Ti=g2i-1g2i} and \eqref{equation gammai=alpha1toh}, we have $g_{2k-1}*g_{2k}=\gamma_k=\alpha$ and $g_{25}*g_{2k}\neq \alpha$.
Denote $T_k'=g_{25}\cdot g_{2k}$ and $W'=g_{2k-1}\cdot g_{26}$.  Recall \eqref{equation g25notinH}, $g_{25}\notin H$. By \eqref{equation T1to Tr1out G}, we have $g_{2k}\notin H$, and so $g_{25}*g_{2k}\in H$. Consider the factorization $S=T_1\cdot\ldots\cdot T_{12}\cdot T_k^{[-1]}\cdot T_k'\cdot W'$ and the corresponding sequence $\tilde{S}'=\gamma_1\bm\cdot\ldots\bm\cdot\gamma_{12}\bm\cdot\gamma_k^{[-1]}\bm\cdot(g_{25}*g_{2k})\in \mathcal{F}(H)$.
It follows from \eqref{equation either or} that $\mathsf v_{\alpha}(\tilde{S}')\in\{7, 4\}$. By the arbitrariness of choosing the factorization in \eqref{equation factorization of Sinto1tot} and choosing $\gamma_i$ in \eqref{equation tilde(S)}, we derive a contradiction with  \eqref{equation either or}.
This proves \eqref{equation g1====g2r-1}.


Next we show that
\begin{equation}\label{equation in casethe sequenceinH}
T_{\mathsf h(\tilde{S})+1}\bm\cdot T_{\mathsf h(\tilde{S})+2}\bm\cdot\ldots\bm\cdot T_{12}\bm\cdot g_{26}\in \mathcal{F}(H).
\end{equation}
 Suppose to the contrary that there exists some term $b$ of $T_{\mathsf h(\tilde{S})+1}\bm\cdot T_{\mathsf h(\tilde{S})+2}\bm\cdot\ldots\bm\cdot T_{12}\bm\cdot g_{26}$ such that $b\notin H$. Recall \eqref{equation g25notinH}, $g_{26}\in H$. Recall \eqref{equation Ti out or in}, there exists some $k\in [h(\tilde{S})+1, 12]$ such $T_{k}\in \mathcal{F}(G\setminus H)$, i.e.,
\begin{equation}\label{equation g2k-1g2knotin H new1}
g_{2k-1},g_{2k}\notin H.
\end{equation}
 Note that $g_{2k-1}\neq g_{25}$ or $g_{2k}\neq g_{25}$, since otherwise, $\gamma_k=g_{2k-1}*g_{2k}=g_{25}*g_{25}=g_1*g_2=\gamma_1=\alpha$, a contradiction with \eqref{equation gammaj=beta} and \eqref{equation gamma1112notin}. Assume without loss of generality that
\begin{equation}\label{equation in case1g2m-1neqg25}
g_{2k-1}\neq g_{25}.
\end{equation}
 Then we exchange $g_{2k-1}$ with $g_{25}$ in the factorization \eqref{equation factorization of Sinto1tot}, that is, we take $S=T_1\cdot\ldots\cdot T_{12}\cdot T_{k}^{[-1]} \cdot T_{k}'\cdot W'$ where $T_{k}'= g_{25}\cdot g_{2k}$ and $W'=g_{2k-1}\cdot g_{26}$. Combined with \eqref{equation G/HcongC2}, \eqref{equation g25notinH} and \eqref{equation g2k-1g2knotin H new1}, we have $\pi(T_k')\subset H$.
By the arbitrariness of choosing the factorization of $S$ given as in \eqref{equation factorization of Sinto1tot}, \eqref{equation length Ti=2} and \eqref{equation pi(Ti)subsetker=H}, and by \eqref{equation g25notinH} and \eqref{equation g1====g2r-1}, we conclude that $g_1=g_2=g_{2k-1}$ and so $g_{2k-1}=g_{25}$, a contradiction with \eqref{equation in case1g2m-1neqg25}. This proves \eqref{equation in casethe sequenceinH}.

Now we claim that we can assume that \begin{equation}\label{equation alpha=1 in case 1}
\alpha=1.
\end{equation}
Suppose $G$ is given as \eqref{equation G2=structure}. By \eqref{equation g25notinH}, \eqref{equation gammai=alpha1toh} and \eqref{equation g1====g2r-1}, we have $\alpha=\gamma_1=g_1*g_2=g_{25}*g_{25}=1$, done.
Suppose $G$ is given as \eqref{equation G=structure}. By Assumption A, we have
$\alpha=y^{c},$
and thus, $g_1*g_2=y^{c}$. Since $g_1,g_2\notin H$, it follows from \eqref{equation G=structure} that $g_2*g_1=g_1^{-1}*(g_1*g_2)*g_1=g_1^{-1}*y^{c}*g_1=y^{-c}$. By \eqref{equation commutates if iin1toh(S)}, we have that $g_2*g_1=g_1*g_2$ and so $y^{-c}=y^{c}$, i.e., $y^{2c}=1$. Since $\alpha=y^3=1$, it follows that $y^{c}=1$. Therefore, \eqref{equation alpha=1 in case 1} can be always assumed to hold.  \\

By \eqref{equation H iso C3timesC3}, \eqref{equation either or}, \eqref{equation alpha beta in H}, \eqref{equation gammai=alpha1toh}, \eqref{equation gammaj=beta} and \eqref{equation alpha=1 in case 1}, for any $k\in [1,8]$ we can always find a product-one subsequence of $\gamma_1\cdot\ldots\cdot \gamma_{8}$ with length $k$, i.e, we can find a product-one subsequence of $T_1\cdot\ldots\cdot T_{8}$ with length $2k$. By the assumption that $S$ has no product-one subsequence of length 18, we conclude that $T_{9}\bm\cdot T_{10}\bm\cdot T_{11}\bm\cdot T_{12}\bm\cdot g_{26}$ has no nonempty product-one subsequence of even length. Denote $V=T_{9}\bm\cdot T_{10}\bm\cdot T_{11}\bm\cdot T_{12}\bm\cdot g_{26}$. Recalling Remark \ref{Remark}, we have $\mathsf d(H)=\mathsf d(C_3\times C_3)=4$. Since $|V|=9=2\mathsf d(H)+1$, we conclude that $V$ has no product-one subsequence of length less than $5$, since otherwise, we shall find two nonempty disjoint product-one subsequences of $V$ which always produces a nonempty product-one subsequence of $V$ with even length, a contradiction. Since $\mathsf d(H)+1=5$, it follows that every subsequence $U$ of $V$ of length $5$ is product-one.
This implies that all terms of $V$ are equal.  Since $V\in \mathcal{F}(H)$ and $H\cong C_3\times C_3$, every subsequence of $V$ with length $6$ is product-one, a contradiction.

\medskip

\noindent {\bf Case 2.} $r=0$. \\

 In a similar argument as \eqref{equation g1====g2r-1} and \eqref{equation in casethe sequenceinH}, we can prove that
\begin{equation}\label{equation g1==g26}
g_i=g_{26} \mbox{ for each } i\in [1,2\mathsf h(\tilde{S})],
\end{equation}
\begin{equation}\label{equation all equals including 25}
T_{\mathsf h(\tilde{S})+1}\bm\cdot T_{\mathsf h(\tilde{S})+2}\bm\cdot\ldots\bm\cdot T_{12}\bm\cdot g_{25}\in \mathcal{F}(G\setminus H).
\end{equation}

Denote $T=T_{\mathsf h(\tilde{S})+1}\bm\cdot T_{\mathsf h(\tilde{S})+2}\bm\cdot\ldots\bm\cdot T_{12}\bm\cdot g_{25}$. By \eqref{equation length Ti=2} and \eqref{equation h(S)=8or5}, we see that $|T|=2\times (12-h(\tilde{S}))+1\geq 2\times (12-8)+1=9$. We take a product-one subsequence $U$ of $T$ with a maximal length. Since $G\diagup H\cong C_2$, it follows from \eqref{equation all equals including 25} that the length $|U|$ is even, combining with \eqref{equation h(S)=8or5}, we have
 \begin{equation}\label{equation |U|leq 10}
 |U|\leq 2\lfloor\frac{|T|}{2}\rfloor=2\times (12-\mathsf h(\tilde{S}))=24-2\mathsf h(\tilde{S})\leq 24-2\times 5=14.
  \end{equation}
  By Table 1 in \cite{CDS}, we see $\mathsf d(G)\leq 7$.  By the maximality of $|U|$, we have that $|T|-|U|=|T\cdot U^{[-1]}|\leq \mathsf d(G)=7$, and so
   \begin{equation}\label{equation |U|geq}
  |U|\geq |T|-7=2\times (12-\mathsf h(\tilde{S}))+1-7=18-2\mathsf h(\tilde{S}).
    \end{equation}
   Denote $\ell=|U|$. By \eqref{equation |U|leq 10} and \eqref{equation |U|geq}, we have
      \begin{equation}\label{equation 8leq18-ellleq}
   4\leq 18-\ell\leq 2 \mathsf h(\tilde{S}).
       \end{equation}
   By \eqref{equation Ti=g2i-1g2i}, we can write $U=g_{i_1}\bm\cdot \ldots \bm\cdot g_{i_{\ell}}$, where $i_1,i_2,\ldots,i_{\ell}\in [2\mathsf h(\tilde{S})+1,25]$. Since $1\in \pi(U)$, we may assume without loss of generality that $g_{i_1}*g_{i_2}*\cdots *g_{i_{\ell}}=1$.
 Denote $g$ to be the element of $G$ equals to $g_{26}$  (in what follows, we use $g$ instead of $g_{26}$ to avoid confusion caused by subscripts).

Now we show that
$$g*g_{i_1}*g=g_{i_1}$$
Suppose $G$ is given as \eqref{equation G2=structure}. By \eqref{equation all equals including 25}, $g_{i_1}\notin H$. By \eqref{equation g25notinH}, we see $g=g_{26}\in H$. By \eqref{equation G2=structure}, we can check that  $g*g_{i_1}*g=g_{i_1}$, done.
Suppose $G$ is given as \eqref{equation G=structure}. By Assumption A, we have
$\alpha=y^{c}.$ Combined with \eqref{equation Ti=g2i-1g2i}, \eqref{equation gammai=alpha1toh} and \eqref{equation g1==g26}, we have that $g^2=g_{26}*g_{26}=g_1*g_2=\gamma_1=\alpha=y^c$, and so
$g=y^{2c}.$
Since $g_{i_1}\notin H$, it follows from \eqref{equation G=structure} that $g*g_{i_1}*g=y^{2c}*g_{i_1}*y^{2c}=g_{i_1}$. This proves $g*g_{i_1}*g=g_{i_1}$.

Then it follows that $g^{\frac{18-\ell}{2}}*g_{i_1}*g^{\frac{18-\ell}{2}}* (g_{i_2}*\cdots *g_{i_{\ell}})=g_{i_1}*g_{i_2}*\cdots *g_{i_{\ell}}=1$.  By \eqref{equation g1==g26}, we see that $g^{[2\mathsf h(\tilde{S})]}\mid S\bm\cdot U^{[-1]}$. Combined with \eqref{equation 8leq18-ellleq}, we have that  $g^{[\frac{18-\ell}{2}]}\cdot g_{i_1}\cdot g^{[\frac{18-\ell}{2}]}\cdot (g_{i_2}\cdot \ldots \cdot g_{i_{\ell}})$ is a product-one subsequence of $S$ with length $18$, which contradicts to the assumption.

In all cases, we get a contradiction. Thus $S$ has a product-one subsequence of length $|G|$. This completes the proof of the lemma.
\end{proof}

Now we are in a position to prove Theorem \ref{maintheorem}.

\medskip

\noindent {\bf Proof of Theorem \ref{maintheorem}.}
Denote $n=|G|$. We shall two consider two cases according to whether $G$ has a cyclic subgroup of index $2$ or not.

We first assume that $G$ has a cyclic subgroup $H$ of index $2$. We show that $\mathsf E(G)=\frac{3|G|}{2}$ in this case.
Since $G$ is non-cyclic,
it follows from Lemma \ref{QLT} that $\mathsf d(G)\leq \frac{n}{2}.$
If $G$ is abelian, then by Lemma~\ref{Gao}, $\mathsf E(G)=|G|+\mathsf d(G)\leq n+\frac{n}{2}=\frac{3|G|}{2}$, done. Hence, we suppose $G$ is non-abelian.
Since $H$ is normal in $G$ and $G$ is non-cyclic, we have that $G\cong C_{\frac{n}{2}}\rtimes C_2$. Then the conclusion follows from Lemma \ref{ZQZk} immediately.

Now we assume that $G$ has no cyclic subgroup of index $2$.
It remains to show that $\mathsf E(G)<\frac{3|G|}{2}$ in this case.
If $n$ is odd,  then the conclusion follows from Lemma \ref{GLQ} and Remark \ref{Remark} immediately.
Suppose $n$ is even.  By Lemma \ref{knowngroup}, we take a subgroup $K$ of $G$ of index $2$,
which is definitely normal in $G$. If $K\not\cong C_3\times C_3$, recalling $4\nmid |G|$ and so $2\nmid |K|$, then it follows from Lemma \ref{GLQ} that $\mathsf E(K)\leq \frac{3|K|-3}{2}< \frac{3|K|}{2}-1$, and it
follows from Lemma \ref{induction} that $\mathsf E(G)\leq \frac{3|G|}{2}-1$, done. Hence, we suppose $K\cong C_3\times C_3$. If $G$ is abelian, then $G\cong C_3\times C_6$ and the conclusion follows from Remark \ref{Remark}. Hence, we suppose $G$ is non-abelian. Since $|G|=2|K|=18$ and $G$ has no cyclic subgroup of index $2$, by the classification of groups of order $18$, we have that $G$ is isomorphic to either $S_3\times C_3$ or $(C_3\times C_3)\rtimes_{-1} C_2$. Then the conclusion follows from Lemma \ref{S3C3} readily. This completes the proof of Theorem \ref{maintheorem}. \qed  \\

\bigskip

\noindent {\bf Acknowledgements}

\noindent
This work is supported by NSFC (grant no. 12371335, 12271520).

\end{document}